\def\version{01/03/2021 -- version 1}

\documentclass[11pt,a4paper]{amsart}

\usepackage{amssymb,amscd,amsxtra,upref}
\usepackage{fullpage}

\usepackage[scr,scaled=1.1]{rsfso}
\usepackage{mathrsfs}

\usepackage[alphabetic,lite]{amsrefs}

\usepackage{xcolor}
\usepackage[all,poly,web,color,matrix,arrow]{xy}
\newdir{ >}{{}*!/-8pt/@{>}}

\usepackage{rotating}

\usepackage{euflag}

\usepackage{mathrsfs}
\usepackage{stmaryrd}

\usepackage{leftidx}

\usepackage{mathtools}
\usepackage{relsize}

\usepackage{tikz}
\usetikzlibrary{arrows,calc}
\usepackage{tikz-cd}[/tikz/commutative diagrams/math mode=1]
\usepackage{tikz-network}

\usepackage[nomargin,inline]{fixme}
\usepackage[]{todonotes}

\newtheorem{thm}{Theorem}[section]

\newtheorem{prop}[thm]{Proposition}

\theoremstyle{remark}

\newtheorem{rem}[thm]{Remark}

\theoremstyle{definition}

\numberwithin{equation}{section}

\def\leq{\leqslant}
\def\geq{\geqslant}
\def\iso{\cong}

\def\Lra{\Longrightarrow}

\def\({\left(}
\def\){\right)}
\def\<{\left<}
\def\>{\right>}
\def\:{\colon}
\def\.{\cdot}

\def\subset{\subseteq}

\def\bar#1{\overline{#1}}
\def\ph#1{\phantom{#1}}
\def\phi{{\varphi}}
\def\epsilon{{\varepsilon}}

\def\F{\mathbb{F}}

\def\k{\Bbbk}

\def\Mod{\mathbf{Mod}}

\DeclareMathOperator{\Cohom}{Cohom}
\DeclareMathOperator{\Coext}{Coext}
\DeclareMathOperator{\Ext}{Ext}
\DeclareMathOperator{\Hom}{Hom}

\DeclareMathOperator{\Cotor}{Cotor}

\def\Sp{\mathrm{Sp}}

\def\SU{\mathrm{SU}}
\def\U{\mathrm{U}}
\def\MSp{{M\Sp}}
\def\MSU{{M\SU}}
\def\MU{{M\mathrm{U}}}
\def\kO{{k\mathrm{O}}}

\def\StA{\mathcal{A}}
\def\StE{\mathcal{E}}
\def\StP{\mathcal{P}}
\DeclareMathOperator{\Sq}{Sq}

\def\op{{\mathrm{op}}}

\DeclareMathOperator{\Id}{Id}
\def\ev{\mathrm{ev}}
\DeclareMathOperator*{\colim}{colim}
\def\Palgebra{$P$-algebra}
\def\Palgebras{$P$-algebras}
\DeclareMathOperator{\pd}{pd}
\def\Mod{\mathbf{Mod}}
\def\MOD{\mathbf{Mod}}
\def\Comod{\mathbf{Comod}}

\newdir{ >}{{}*!/-8pt/@{>}}

\title[The dual of a $P$-algebra with applications
to comparison of Bousfield classes]
{On the dual of a $P$-algebra and its comodules,
with applications to comparison of some
Bousfield classes}
\date{last updated \version}
\author{Andrew Baker}
\address{
School of Mathematics \& Statistics,
University of Glasgow, Glasgow G12~8QQ, Scotland.}
\email{a.baker@maths.gla.ac.uk}
\urladdr{http://www.maths.gla.ac.uk/$\sim$ajb}
\thanks{
I would like to thank the following: The
Max-Planck-Institut f\"ur Mathematik in
Bonn for supporting my visit in January
2020; Tobias Barthel for raising the
question that motivated this work; Doug
Ravenel for support and encouragement
and teaching me so much over many decades;
Ken Brown and Nige Ray for numerous helpful
conversations on algebra and topology.  \\
I would like to dedicate this paper to the
memory of \textbf{Pete Bousfield} (1941--2020)
whose work on localisation has become a
central feature of modern homotopy theory.
}

\begin{document}

\begin{abstract}
In his seminal work on localisation of spectra,
Ravenel initiated the study of Bousfield classes
of spectra related to the chromatic perspective.
In particular he showed that there were infinitely
many distinct Bousfield classes between
$\langle\MU\rangle$ and $\langle S^0\rangle$.
The main topological goal of this paper is
investigate how these Bousfield classes are
related to that of another classical Thom
spectrum~$\MSp$, and in particular how
$\langle\MSp\rangle$ is related to
$\langle\MU\rangle$.

We follow the approach of Ravenel, but adapt
it using the theory of $P$-algebras to give
vanishing results for cohomology. Our work
involves dualising and considering comodules
over duals of $P$-algebras; these ideas are
then applied to the mod~$2$ Steenrod algebra
and certain subHopf algebras.

\end{abstract}

\maketitle
\tableofcontents

\section*{Introduction}

In his seminal paper on localisation~\cite{DCR:Localn},
Doug Ravenel initiated the study of Bousfield
classes of spectra related to the chromatic
perspective. In particular he showed that
there were infinitely many distinct Bousfield
classes between $\langle\MU\rangle$ and
$\langle S^0\rangle$. The main topological
goal of this paper is investigate how these
Bousfield classes are related to that of
another classical Thom spectrum~$\MSp$,
and in particular how $\langle\MSp\rangle$
is related to $\langle\MU\rangle$. Of course,
there is a canonical map of ring spectra
$\MSp\to\MU$ so
$\langle\MSp\rangle\geq\langle\MU\rangle$.
This issue is mentioned by Hovey
\& Palmieri~\cite{MH&JHP:BousfieldLattice},
who also say that Ravenel had outlined an
argument that $\langle\MSp\rangle>\langle\MU\rangle$
but this has not appeared in print.

To achieve this we rework the algebraic
machinery in Ravenel's approach making use
of the theory of \Palgebras{} of Moore \&
Peterson~\cite{JCM-FPP:NearlyFrobAlgs},
subsequently developed further by
Margolis~\cite{HRM:Book}. We first consider
the dual of a \Palgebra{} and its comodules;
in order to set up some Cartan-Eilenberg
spectral sequences we discuss the homological
algebra required to make a bivariant derived
functor of the homomorphism for comodules
with the aim of obtaining three such spectral
sequences, one of which involves dualising
to modules over the \Palgebra{} itself because
comodules do not always have resolutions by
projective objects. After reviewing some
properties of the mod~$2$ Steenrod algebra
and its dual, we discuss doubling and then
describe some families of subHopf algebras
and their duals. Moving on to topology, we
introduce a sequence of loop spaces whose
colimit is equivalent to $B\Sp$ and which
give rise to a family of Thom spectra whose
colimit is equivalent to~$\MSp$. After
describing the homology of these as comodule
algebras we move on to our main result which
says that $2$-locally,
$\langle\MSp\rangle>\langle BP\rangle$
and there is an infinite sequence of
distinct Bousfield classes between
$\langle S^0\rangle$ and $\langle\MSp\rangle$.

\section{\Palgebras{} and their duals}
\label{sec:P-algebras}
We will make use material on \Palgebras{} and their
modules contained in~\cite{HRM:Book}*{chapter~13}.
Here is a summary of assumptions and conventions.
\begin{itemize}
\item
We will often suppress explicit mention of
internal grading in cohomology and just
write $M$ for $M^*$ when discussing a module
over $A=A^*$. When working in homology we will
usually write $M_*$ or $A_*$.
\item
We will work with graded vector spaces over
a field~$\k$ and in particular, $\k$-algebras
and their (co)modules. In our topological
applications,~$\k=\F_2$.

For a finite type graded vector space $V_*$ we think
of~$V_n$ as dual to $V^n=\Hom_\k(V_n,\k)$, so $V^*$
is the cohomologically graded degree-wise dual, and
bounded below means that $V^*$ is bounded below. We
can also start with a finite type graded vector space
$V^*$ and form $V_*$ where $V_n=\Hom_\k(V^n,\k)$;
the double dual of $V_*$ is canonically isomorphic
to $V_*$, and vice versa.

A graded vector space which is finite dimensional
will be referred to as a \emph{finite}; when~$\k$
is finite this terminology agrees with Margolis'
use of finite module over a \Palgebra.
\item
Recall that a \Palgebra~$A$ is a union
of connected cocommutative Hopf algebras
$A(n)\subset A(n+1)\subset A$. Each $A(n)$
is a Poincar\'e duality algebra and we
will denote its highest non-trivial
degree by~$\pd(n)$, which is also its
Poincar\'e duality degree; this satisfies
$\pd(n)<\pd(n+1)$. In general, \Palgebras{}
are not required to be of finite type,
but all the examples we consider have
that property and it is required in some
of our homological results so we will
assume it holds. We also stress that
each $A(n)$ is a local graded ring
and $A(n)^{\pd(n)}=A(n)^0=\k$. Other
important properties are that~$A$ is
flat as a left or right $A(n)$-module
and a coherent $\k$-algebra.

Of course our main examples are the
Steenrod algebra~$\StA$ and infinite
dimensional subHopf algebras.
\end{itemize}

We will use the following basic result
stated on page~195 of~\cite{HRM:Book}
but left as an exercise.
\begin{prop}\label{prop:Palg-finite}
Suppose that $A$ is a P-algebra. Let~$M$
be a finite $A$-module and~$F$ a bounded
below free $A$-module. Then
\[
\Ext_A^*(M,F)=0.
\]
\end{prop}
\begin{proof}
By~\cite{HRM:Book}*{theorem~13.12}, bounded
below projective $A$-modules are also injective,
so $\Ext_A^s(M,F)=0$ when $s>0$, therefore
we only have to show that $\Hom_A(M,F)=0$.
It suffices to prove this for the case~$F=A$.

Suppose that $0\neq\theta\in\Hom_A(M,A)$.
The image of $\theta$ contains a simple
submodule in its top degree, so let
$\theta(x)\neq0$ be in this submodule;
then $a\theta(x)=0$ for every positive
degree element $a\in A$. Now for some~$n$,
$\theta(x)\in A(n)^k$ where $k<\pd(n)$,
and by Poincar\'e duality for~$A(n)$
there exists $z\in A(n)$ for which
$0\neq z\theta(x)\in A(n)^{\pd(n)}$.
This gives a contradiction, hence no
such~$\theta$ can exist.
\end{proof}

A particular concern for us will be the
situation where we have a \emph{pair}
of \Palgebras{} $B\subseteq A$ (so~$B$
is a subHopf algebra of~$A$).
\begin{prop}\label{prop:ExtA/B}
For a pair of P-algebras $B\subseteq A$,
\[
\Ext^*_A(A\otimes_B\k,A) \iso \Ext^*_B(\k,A) = \Hom_B(\k,A) = 0.
\]
\end{prop}
\begin{proof}
First we recall a classic result of Milnor \&
Moore~\cite{M&M:HopfAlg}*{proposition~4.9}:
$A$ is free as a left or right $B$-module.
This guarantees the change of rings isomorphism
and the second isomorphism follows from
Proposition~\ref{prop:Palg-finite}.
\end{proof}

It is crucial that~$B$ is it self a \Palgebra,
for example the vanishing result does not hold
if $B=A(n)$.

Since a \Palgebra{} $A$ is coherent, its finitely
generated modules are its coherent modules, and
they form an abelian category with finite limits
and colimits (in particular it has kernels,
cokernels and images). A coherent $A$-module~$M$
admits a finite presentation
\[
\xymatrix{
A^{\oplus k}\ar[r]^{\pi}& A^{\oplus \ell}\ar[r]& M\ar[r]&0
}
\]
which can be defined over some $A(n)$, i.e.,
there is a finite presentation
\[
\xymatrix{
A(n)^{\oplus k}\ar[r]^{\pi'}&A(n)^{\oplus \ell}\ar[r]& M'\ar[r]&0
}
\]
of $A(n)$-modules and a commutative diagram
of $A$-modules
\[
\xymatrix{
A\otimes_{A(n)}A(n)^{\oplus k}\ar[r]^{\Id\otimes\pi'}\ar@{->}[d]
& A\otimes_{A(n)}A(n)^{\oplus \ell}\ar[r]\ar@{<->}[d]^\iso
& A\otimes_{A(n)}M'\ar[r]\ar@{<->}[d]^\iso &0 \\
A^{\oplus k}\ar[r]^{\pi}& A^{\oplus \ell}\ar[r]& M\ar[r]&0
}
\]
with exact rows. It is standard that every
coherent $A$-module admits a resolution by
finitely generated free modules. But for a
\Palgebra{} we also have injective resolutions
by finitely generated free modules.
\begin{prop}\label{prop:Coherent-InjResn}
Let $M$ be a coherent module over a \Palgebra{}
$A$. Then $M$ admits an injective resolution
by finitely generated free modules.
\end{prop}
\begin{proof}
By~\cite{HRM:Book}*{theorem~13.12}, bounded
below projective $A$-modules are injective.

For some $n$, $M\iso A\otimes_{A(n)}M'$
where~$M'$ is a finitely generated $A(n)$-module.
It is standard that~$M'$ admits a monomorphism
$M'\to J'$ into a finitely generated free
module (since~$A(n)$ is a Poincar\'e duality
algebra this is obvious for a simple module
and can be proved by induction on dimension).
By flatness we can find a monomorphism of
$A$-modules
\[
M\xrightarrow{\iso} A\otimes_{A(n)}M'\to A\otimes_{A(n)}J'
\]
into a finitely generated free module with coherent
cokernel. Iterating this we can build an injection
resolution of the stated form.
\end{proof}

\begin{prop}\label{prop:Finite->coherent}
Let $M$ be a finite $A$-module and $N$ a coherent
$A$-module. Then
\[
\Ext_A^*(M,N) = 0.
\]
\end{prop}
\begin{proof}
By Proposition~\ref{prop:Coherent-InjResn} there
is an injective resolution $N\to J^*$ where
each $J^s$ is a finitely generated free module.
Also, by Proposition~\ref{prop:Palg-finite}
$\Hom_A(M,J^s) = 0$ and so $\Ext_A^s(M,N) = 0$.
\end{proof}

For example, every left $A$-module of the form
$A/\!/A(n)=A\otimes_{A(n)}\k$ is coherent so
it admits such an injective resolution and
$\Ext_A^s(\k,A/\!/A(n)) = 0$.

\subsection*{The dual of a \Palgebra{} and its comodules}
The theory of \Palgebras{} can be dualised: given
a finite type \Palgebra~$A$, we define its dual
$A_*$ by setting $A_n=\Hom_\k(A^n,\k)$ and making
this a commutative Hopf algebra. We will refer to
the dual Hopf algebra of a finite type \Palgebra{}
as a \emph{$P_*$-algebra}; as far as we know this
is not standard terminology but it seems appropriate.

When working with comodules over a $P_*$-algebra~$A_*$
we will use homological grading. For left
$A_*$-comodules which are bounded below and
of finite type there is no significant difference
between working with them or their (degree-wise)
duals as $A$-modules. In particular,
\[
 \Cohom_{A_*}(M_*,N_*) \iso \Hom_{A}(N^*,M^*),
\]
where $M^n=\Hom_{\F_2}(M_n,\F_2)$ and $M^*$ is
made into a left $A$-module using the antipode.
More generally,
\[
\Coext^{s,*}_{A_*}(M_*,N_*) \iso \Ext^{s,*}_{A}(N^*,M^*),
\]
where $\Coext^{s,*}_{A_*}(M_*,-)$ denotes the
right derived functor of $\Cohom_{A_*}(M_*,-)$,
which can be computed using extended comodules
which are injective comodules here since we
are working over a field.

Here are the dual versions of
Propositions~\ref{prop:Palg-finite}
and~\ref{prop:ExtA/B}. Recall that a
cofree $A_*$-comodule is one of the
form $A_*\otimes W_*$.
\begin{prop}\label{prop:P*alg-finite}
Suppose that $A_*$ is a $P_*$-algebra.
Let~$M_*$ be a finite $A_*$-comodule
and~$C_*$ a bounded below finite type
cofree $A_*$-comodule. Then
\[
\Coext_{A_*}^*(C_*,M_*)=0.
\]
\end{prop}
\begin{prop}\label{prop:CoextA/B}
For a surjective morphism of Hopf algebras
$A_*\to B_*$ dual to a pair of P-algebras
$B\subseteq A$,
\[
\Coext^{*,*}_{A_*}(A_*,A_*\square_{B_*}\k)
\iso
\Coext^{*,*}_{B_*}(A_*,\k) = \Cohom_{B_*}(A_*,\k) = 0.
\]
\end{prop}

If $A_*$ is a $P_*$-algebra then we will call
an $A_*$-comodule~$M_*$ \emph{coherent} if its
dual $M^*$ is a coherent $A$-module. Here is
a dual version of Proposition~\ref{prop:Coherent-InjResn}.
\begin{prop}\label{prop:CoherentComod-ProjResn}
Let $M_*$ be a coherent comodule over a $P_*$-algebra{}
$A_*$. Then~$M_*$ admits a resolution by finitely
generated cofree comodules.
\end{prop}
\begin{proof}
Take an injective resolution of $M^*$ as
in Proposition~\ref{prop:Coherent-InjResn}
and then take duals to obtain a projective
resolution.
\end{proof}

Notice that since $A_*$ is an injective comodule
we have
\[
\Coext_{A_*}^*(A_*,A_*) \iso \Hom_\k(A_*,\k).
\]

\begin{prop}\label{cor:CoherentComod-ProjResn}
Let $M_*$ be a coherent comodule over a $P_*$-algebra{}
$A_*$ and let $N_*$ be a finite $P_*$-comodule.
Then
\[
\Coext_{A_*}^*(M_*,N_*)=0.
\]
\end{prop}
\begin{proof}
Let $P_{\bullet,*}\to M_*\to 0$ be a resolution
of~$M_*$ by cofree comodules. Then by
Proposition~\ref{prop:P*alg-finite}, for each
$s\geq0$ we have
\[
\Cohom_{A_*}(P_{s,*},N_*) = 0,
\]
and the result follows.
\end{proof}

\begin{rem}\label{rem:comodules-modules}
For bounded below finite type comodules over
a $P_*$-algebra $A_*$ dual to a \Palgebra{}
$A$, taking degree-wise duals defines an
equivalence of categories
\[
\xymatrix{
\Comod_{A_*}^{\flat,\,\mathrm{f.t.}}\ar@<.5ex>[rr]^(.5){(-)^*}
&& (\Mod_{A}^{\flat,\,\mathrm{f.t.}}\ar@<.5ex>[ll]^(.5){(-)_*})^\op
}
\]
between the $A_*$-comodule and the $A$-module
categories. Moreover, these functors are exact,
so this equivalence identifies injective
comodules (which are retracts of extended
comodules) with projective modules.
By~\cite{HRM:Book}*{theorem~13.12}, projective
$A$-modules are injective so it also identifies
injective comodules as projective objects (this
is not true in general).

In fact this equivalence fits into a bigger
diagram
\begin{equation}\label{eq:Comod->Mod}
\xymatrix{
&& \Mod_{A}^{\sharp,\,\mathrm{f.t.}}\ar@<.5ex>[d]^(.5){(-)^*} \\
\Comod_{A_*}^{\flat,\,\mathrm{f.t.}}
\ar@<.5ex>[rr]^(.5){(-)^*}
\ar@/^15pt/[urr] \ar@/_15pt/[drr]
&& (\Mod_{A}^{\flat,\,\mathrm{f.t.}})^\op
\ar@<.5ex>[ll]^(.5){(-)_*}\ar@<.5ex>[u]^(.5){(-)_*}\ar[d]  \\
&& \MOD_{A}^\op
}
\end{equation}
where $\Mod_{A}^{\natural,\,\mathrm{f.t.}}$
denotes the category of finite type bounded
below homologically graded $A$-modules
(with $A$ acting by decreasing degree),
$\Mod_{A}^{\flat,\,\mathrm{f.t.}}$ denotes
the category of finite type bounded below
cohomologically graded $A$-modules
and $\MOD_{A}$ denoting the category
of all $A$-modules. All of the functors
here are exact.

For a fixed $A_*$-comodule $M_*$, the
functor
\[
\Cohom_{A_*}(M_*,-)=\Comod_{A_*}^{\flat,\,\mathrm{f.t.}}(M_*,-)
\to\Mod_\k^{\flat,\,\mathrm{f.t.}}
\]
is left exact and has right derived functors
$\Coext^*_{A_*}(M_*,-)$. Since
\[
\Cohom_{A_*}(M_*,-)\iso\Hom_{A}((-)^*,M^*)
=\Mod_{A}^{\flat,\,\mathrm{f.t.}}((-)^*,M^*)
=\MOD_{A}((-)^*,M^*)
\]
and injective comodules are sent to projective
modules, we also have
\begin{equation}\label{eq:Coext-Ext}
\Coext^*_{A_*}(M_*,-)\iso \Ext^*_{A}((-)^*,M^*).
\end{equation}

The contravariant functor
$\Comod_{A_*}^{\flat,\,\mathrm{f.t.}}\to\MOD_{A}^\op$
allows us to define cohomological invariants
of comodules using injective resolutions
in $\MOD_{A}$ as a substitute for projective
resolutions in
$\Comod_{A_*}^{\flat,\,\mathrm{f.t.}}$.
In effect for a comodule~$N_*$ we define
\[
\Coext^*_{A_*}(-,N_*)=\Ext_A(N^*,(-)^*).
\]
Of course this is calculated using injective
resolutions of $A$-modules; since $\Ext_A(-,-)$
is a balanced functor,~\eqref{eq:Coext-Ext}
implies that $\Coext^*_{A_*}(-,-)$ is too,
whenever we can use projective comodule
resolutions in the first variable. For
example, if we restrict to the subcategory
of coherent comodules we obtain balanced
bifunctors
\[
\Coext^s_{A_*}(-,-)\:
(\Comod_{A_*}^{\mathrm{coh}}(-,-))^{\op}
\otimes\Comod_{A_*}^{\mathrm{coh}}(-,-)
\to \Mod_\k^{\flat,\,\mathrm{f.t.}}.
\]
\end{rem}

Given a surjection of $P_*$-algebras
$A_*\to B_*$ there are
adjunction isomorphisms of the form
\begin{align}\label{eq:Cohom-Cotor-1}
\Cohom_{A_*}(-,-)
&\iso
\Cohom_{B_*\backslash\!\backslash A_*}((B_*\backslash\!\backslash A_*)\square_{A_*}(-),-),
 \\ \label{eq:Cohom-Cotor-2}
\Cohom_{B_*}(-,-)
&\iso
\Cohom_{A_*}(-,A_*\square_{B_*}(-)),
\end{align}
which we will use to construct Grothendieck
spectral sequences.

\section{Some homological algebra for comodules}
\label{sec:HomAlg}

In this section we describe some Cartan-Eilenberg
spectral sequences for comodules over a commutative
Hopf algebra over a field. Some of these are
similar to other examples in the literature
such as that for computing~$\Cotor$ for Hopf
algebroids in~\cite{DCR:GreenBook}. However,
the spectral sequence that plays the biggest
r\^ole in our calculations is not defined
in the same way but requires dualising and
is a standard Cartan-Eilenberg spectral
sequence for $\Ext$ of modules over a ring.

To ease notation, in this section we suppress
internal gradings and assume that all our
objects are connective and of finite type
over a field~$\k$. We refer to the classic
\cite{M&M:HopfAlg} as well as the more
recent \cite{JPM&KP:MoreConcise} for notation
and basic ideas about graded Hopf algebras.

Suppose we have a sequence of homomorphisms
of connected commutative graded Hopf algebras
over~$\k$,
\[
K\backslash\!\backslash H \rightarrowtail H \twoheadrightarrow K,
\]
where in the notation of \cite{M&M:HopfAlg}*{definition~3.5},
\[
K\backslash\!\backslash H
= \k\square_K H
= H \square_K\k
\subseteq H .
\]
We also assume given a left
$K\backslash\!\backslash H$-comodule
$M$ and a left $H $-comodule~$N$. Of course
$M$ and $N$ inherit structures of $H $-comodule
and $K$-comodule respectively, where~$M$
is trivial as a $K$-comodule. Our aim is
to calculate $\Coext^*_H (M,N)$, the right
derived functor of
\[
 Cohom_H (M,-)\:
\Comod_{K\backslash\!\backslash H}\to\Mod_\k;
\quad N\mapsto \Cohom_H (M,N).
\]

Following Hovey~\cite{MH:HtpyThyComods}, we
will write $U\overset{H }\wedge V$ to indicate
the tensor product of two $H $-comodules
$U\otimes V=U\otimes_\k V$ with the diagonal
coaction given by the composition

\[
\xymatrix{
U\otimes V\ar[r]_(.3){\mu\otimes\mu}\ar@/^18pt/[rrr]
& (H \otimes U)\otimes (H \otimes V)\ar[r]_{\iso}
& (H \otimes H)\otimes (U\otimes V)\ar[r]_(.55){\phi\otimes\Id}
& H \otimes (U\otimes V).
}
\]
For a vector space $W$, the notation $U\otimes W$
will be used to denote $H $-comodule with
coaction

\[
\xymatrix{
U\otimes W\ar[r]_(.4){\mu\otimes\Id}\ar@/^14pt/[rr]
& (H \otimes U)\otimes W\ar[r]_(.5){\iso}
& H \otimes(U\otimes W)
}
\]
carried on the first factor alone.

If $L$ is a left $H $-comodule, then
there is a well-known isomorphism of left
$H $-comodules
\begin{equation}\label{eq:Comodtwisting}
K\backslash\!\backslash H \overset{H }\wedge L
=
(H \square_K\k)\overset{H }\wedge L
\iso
H \square_K L,
\end{equation}
We can also regard
$K\backslash\!\backslash H =\k\square_K H $
as a right $H $-comodule to form the left
$K\backslash\!\backslash H $-comodule
\begin{equation}\label{eq:Comodtwisting-Sigma}
K\backslash\!\backslash H \square_H  L
= (\k\square_K H )\square_H  L
\iso \k\square_K L;
\end{equation}
in particular, if $L$ is a trivial $K$-comodule
then as left $K\backslash\!\backslash H$-comodules,
\begin{equation}\label{eq:Comodtwisting-trivial}
K\backslash\!\backslash H \square_H  L
\iso L.
\end{equation}

We will use two more functors
\[
\Comod_{H }\to\Comod_{K\backslash\!\backslash H};
\quad
N\mapsto
K\backslash\!\backslash H\square_H  N
= (\k\square_K H)\square_H  N
\iso \k\square_K N
\]
and
\[
\Comod_{K\backslash\!\backslash H}\to\Mod_\k;
\quad
N\mapsto  \Cohom_{K\backslash\!\backslash H}(M,N).
\]
Notice that there is a natural isomorphism
\[
 \Cohom_{K\backslash\!\backslash H}(M,K\backslash\!\backslash H\square_H (-))
\iso  \Cohom_H (M,-)
\]
and for an injective $H $-comodule~$J$,
$K\backslash\!\backslash H \square_H  J$
is an injective $K\backslash\!\backslash H$-comodule.
This means we are in a situation where we
have a Grothendieck composite functor spectral
sequence which in this case is a form of
Cartan-Eilenberg spectral sequence; for details
see \cite{CAW:HomAlg}*{section~5.8} for example.

\begin{prop}\label{prop:CESS-1}
Let $M$ be a left\/
$K\backslash\!\backslash H$-comodule
and $N$ a left\/ $H $-comodule. Then
there is a first quadrant cohomologically
indexed spectral sequence with
\[
\mathrm{E}_2^{s,t} =
\Coext_{K\backslash\!\backslash H}^s(M,\Cotor_K^t(\k,N))
\Lra
\Coext_H ^{s+t}(M,N).
\]
If $N$ is a trivial\/ $K$-comodule then
\[
\mathrm{E}_2^{s,t} \iso
\Coext_{K\backslash\!\backslash H}^s(M,
\Cotor_K^t(\k,\k)\overset{K\backslash\!\backslash H}\wedge N).
\]
\end{prop}

There is another spectral sequence that
we will use whose construction requires
that one of the Hopf algebras involved
is a $P_*$-algebra. The reason for this
is discussed in Remark~\ref{rem:comodules-modules}:
in the category of finite type connected
comodules, extended comodules are projective
objects.

\begin{prop}\label{prop:CESS-2}
Assume that $H $ and $K\backslash\!\backslash H$
are $P_*$-algebras. Let~$M$ be a left\/
$H $-comodule which admits a projective
resolution and let $N$ be a left\/
$K\backslash\!\backslash H$-comodule.
Then there is a first quadrant cohomologically
indexed spectral sequence with
\[
\mathrm{E}_2^{s,t} =
\Coext_{K\backslash\!\backslash H}^s(\Cotor_K^t(\k,M),N)
\Lra
\Coext_H ^{s+t}(M,N).
\]
If $M$ is a trivial\/ $K$-comodule then
\[
\mathrm{E}_2^{s,t} \iso
\Coext_{K\backslash\!\backslash H}^s
(\Cotor_K^t(\k,\k)\overset{K\backslash\!\backslash H}\wedge M,N).
\]
\end{prop}
\begin{proof}
The construction is similar to the other one,
and involves expressing $\Cohom_H (-,N)$
as a composition
\[
\Cohom_{K\backslash\!\backslash H}(-,N)\circ (K\backslash\!\backslash H\square_H (-))
=
\Cohom_{K\backslash\!\backslash H}(K\backslash\!\backslash H\square_H (-),N)
\iso \Cohom_H (-,N).
\]
The functor $K\backslash\!\backslash H \square_H (-)$
sends injective $H$-comodules to projective
objects in $\Comod_{K\backslash\!\backslash H }$
(see Remark~\ref{rem:comodules-modules}).
Therefore the standard construction can be
applied.
\end{proof}

Of course the assumption that $M$ admits a
projective resolution is crucial; in the
case of $P_*$-algebras this amounts to
working with coherent comodules.

Unfortunately, these spectral sequences
are not sufficient for our purposes and
we also need to dualise and use a classical
Cartan-Eilenberg spectral sequence
\cite{HC&SE:HomAlg}*{theorem~6.1(1)} for
a normal sequence of $\k$-algebras
\begin{equation}\label{eq:CESS-algebraseq}
R\to S \to S/\!/R
\end{equation}
together with a left $S/\!/R$-module~$L$
and a left $S$-module~$M$. This spectral
sequence has the form
\begin{equation}\label{eq:CESS-algebras}
\mathrm{E}_2^{s,t} =
\Ext^s_{S/\!/R}(L,\Ext^t_{R}(\k,M))
\Lra \Ext^{s+t}_{S}(L,M),
\end{equation}
where we have suppressed internal gradings.
In our applications~\eqref{eq:CESS-algebraseq}
will be a sequence of cocommutative Hopf
algebras.

\section{The Steenrod algebra and its dual}
\label{sec:StA}

The theory of \Palgebras{} applies to many
situations involving subHopf algebras of the
Steenrod algebra for a prime. Of course the
change of rings isomorphism of Proposition~\ref{prop:ExtA/B}
holds for any subalgebra where~$A$ is flat
over~$B$, but the vanishing will not be true
in general (for example, when $B=A(n)$).

To illustrate this, here is a simple application
involving the mod~$2$ Steenrod algebra; this
result appears in~\cite{DCR:Localn}*{corollary~4.10}.
We denote the mod~$2$ Eilenberg-Mac~Lane
spectrum by~$H=H\F_2$.
\begin{prop}\label{prop:H->BP}
The $2$-completed $BP$-cohomology of $H$ is
trivial, i.e., $(BP\sphat_2)^*(H)=0$.
\end{prop}
\begin{proof}
There is an Adams spectral sequence of form
\[
\mathrm{E}_2^{s,t} = \Coext_{\StA_*}^{s,t}(H_*(H),H_*(BP))
\Lra (BP\sphat_2)^{s-t}(H).
\]
Now $H_*(H)=\StA_*$ and it is well-known
that as $\StA_*$-comodule algebras,
\begin{equation}\label{eq:H*BP}
H_*(BP)
\iso \StA_*\square_{\StE_*}\F_2
= \StA_*^{(1)}
= \F_2[\zeta_1^2,\ldots,\zeta_n^2,\ldots]
\end{equation}
where
$\StE_* = \StA_*/\!/\StA_*^{(1)}$ which
is an infinitely generated exterior Hopf
algebra so it is a $P_*$-algebra{}. Now
by a change of rings isomorphism and
Proposition~\ref{prop:ExtA/B},
\[
\mathrm{E}_2^{s,t} =
\Coext_{\StA_*}^{s,t}(\StA_*,\StA_*\square_{\StE_*}\F_2)
\iso
\Coext_{\StE_*}^{s,t}(\StA_*,\F_2)
= 0.
\qedhere
\]
\end{proof}

\subsection*{Doubling}

The operation of \emph{doubling} has been
used frequently in studying $\StA$-modules.
The reader is referred to the account of
Margolis~\cite{HRM:Book}*{section~15.3}
which we will use as background.

Since the dual $\StA_*$ is a commutative
Hopf algebra, it admits a Frobenius
endomorphism $\StA_*\to\StA_*$ which
doubles degrees and has Hopf algebra
cokernel
\[
\mathcal{E}_* = \StA_*/\!/\StA_*^{(1)}
= \Lambda_{\F_2}(\bar{\zeta}_s : s\geq1),
\]
where $\StA_*^{(1)}=\F_2[\zeta_s^2 : s\geq1 ]$.
Dually, there is a Verschiebung $\StA\to\StA$
which halves degrees and satisfies
\[
\Sq^r\mapsto
\begin{dcases*}
\Sq^{r/2} & if $r$ is even, \\
\ph{\;\;\;}0 & if $r$ is odd.
\end{dcases*}
\]
The kernel of this Verschiebung is the
ideal generated by the Milnor primitives
$\mathrm{P}^0_t$ ($t\geq1$), hence there
is a grade-halving isomorphism of Hopf
algebras $\StA/\!/\mathcal{E}\xrightarrow{\iso}\StA$,
where $\mathcal{E}\subseteq\StA$ is the
subHopf algebra generated by the primitives
$\mathrm{P}^0_t$ and dual to the exterior
quotient Hopf algebra $\mathcal{E}_*$.

Given a left (graded) $\StA$-module~$M$, we
can induce an $\StA/\!/\mathcal{E}$-module
$M_{(1)}$ where
\[
M_{(1)}^n =
\begin{dcases*}
M^{n/2} & if $n$ is even, \\
0 & if $n$ is odd,
\end{dcases*}
\]
and we write $x_{(1)}$ to indicate the
element $x\in M$ regarded as an element
of $M_{(1)}$; the module structure is
given by
\[
\Sq^r(x_{(1)}) =
\begin{dcases*}
(\Sq^{r/2}x)_{(1)} & if $r$ is even, \\
0 & if $r$ is odd,
\end{dcases*}
\]

Using this construction, the category of
left $\StA$-modules $\Mod_{\StA}$ admits
an additive functor to the category of
evenly graded $\StA/\!/\mathcal{E}$-modules,
\[
\Phi\:
\Mod_{\StA} \to
\Mod_{\StA/\!/\mathcal{E}}^\ev;
\quad
M\mapsto M_{(1)}
\]
which is an isomorphism of categories.
The quotient homomorphism $\rho\:\StA\to\StA/\!/\mathcal{E}$
also induces an additive isomorphism of categories
$\rho^*\:\Mod_{\StA/\!/\mathcal{E}}^\ev
\to\Mod_{\StA}^\ev$
and it is often useful to consider the
composition
$\rho^*\circ\Phi\:\Mod_{\StA}
\to \Mod_{\StA}^\ev$.

By iterating $\Phi^{(1)} = \Phi$ we obtain
isomorphisms
\[
\Phi^{(s)} = \Phi\circ\Phi^{(s-1)}\:\Mod_{\StA}
\to
\Mod_{\StA/\!/\StE^{(s)}}^{(s)};
\quad
M\mapsto M_{(s)}
\]
where the codomain is the category of
$\StA/\!/\StE^{(s)}$-modules concentrated
in degrees divisible by~$2^s$ and
$\StE^{(s)}\subseteq\StA$ is the subHopf
algebra multiplicatively generated by the
elements
\[
\mathrm{P}^a_b \quad (s\geq a\geq0,\;b\geq1).
\]

By doubling all three of the variables involved
the following homological result is immediate
for $e\geq1$ and two $\StA$-modules~$M,N$:
\begin{align}\label{eq:Doubling-iterated-Ext}
\Ext_{\StA_{(e)}}^{s,2^et}(M_{(e)},N_{(e)})
\iso \Ext_\StA^{s,t}(M,N).
\end{align}

Because doubling is induced using a grade
changing Hopf algebra endomorphism, the
double $\StA_{(1)}$ is also a Hopf algebra
isomorphic to the quotient Hopf algebra
$\StA/\!/\mathcal{E}$ and dual to the
subHopf algebra of squares
$\StA^{(1)}_*\subseteq\StA_*$ which is
also given by
\[
\StA^{(1)}_*
= \StA_*\square_{\StA_*/\!/\StA^{(1)}_*}\F_2
= \F_2\square_{\StA_*/\!/\StA^{(1)}_*}\StA_*
= (\StA_*/\!/\StA^{(1)}_*)\backslash\!\backslash\StA_*.
\]
More generally, for any $s\geq1$, $\StA_{(s)}$
is isomorphic to the quotient Hopf algebra
of~$\StA/\!/\StE^{(s)}$ dual to the subalgebra
of $2^s$-th powers
\[
\StA^{(s)}_*
= (\StA_*/\!/\StA^{(s)}_*)\backslash\!\backslash\StA_*
\subseteq\StA_*.
\]

In many ways, doubling is more transparent
when viewed in terms of comodules. For
an $\StA_*$-comodule $M_*$, we can define
a $\StA_*^{(1)}$-coaction
$\mu^{(1)}\colon M_*^{(1)}\to\StA_*^{(1)}\otimes M_*^{(1)}$
where $M_*^{(1)}$ denotes $M_*$ with its
degrees doubled; this is given on elements
by the composition
\[
\xymatrix{
M_*\ar[rr]_(.45)\mu \ar@/^15pt/[rrrr]^{\mu^{(1)}}
&& \StA_*\otimes M_*\ar[rr]_(.45){(-)^2\otimes\Id}
&& \StA_*^{(1)}\otimes M_*.
}
\]
By iterating we also obtain a $\StA_*^{(s)}$-coaction
$\mu^{(s)}\colon M_*^{(s)}\to\StA_*^{(s)}\otimes M_*^{(s)}$.

Then the comodule analogue of~\eqref{eq:Doubling-iterated-Ext}
is
\begin{align}\label{eq:Doubling-iterated-Coxt}
\Coext_{\StA_*^{(e)}}^{s,2^et}(M_*^{(e)},N_*^{(e)})
\iso \Coext_{\StA_*}^{s,t}(M_*,N_*).
\end{align}

We can use iterated doubling combined with
Proposition~\ref{prop:ExtA/B} to show that
for any $d\geq1$,
\begin{equation}\label{eq:Doubling}
\Coext_{\StA_*}^{s,t}(\StA_*,\StA_*^{(d)})
\iso \Ext_\StA^{s,t}(\StA_{(d)},\StA)
= 0.
\end{equation}
By doubling all three of the variables involved
here we can also prove that for $e\geq0$,
\begin{equation}\label{eq:Doubling-iterated-comodules}
\Coext_{\StA_*^{(e)}}^{s,2^et}(\StA_*^{(e)},\StA_*^{(d+e)})
\iso \Coext_{\StA_*}^{s,t}(\StA_*,\StA_*^{(d)})
= 0.
\end{equation}

We will use this to show that $(\MSp\sphat_2)^*(BP)=0$
and $(\MSp\sphat_2)^*(H)=0$ for example.

\subsection*{Some families of quotient $P_*$-algebras
of $\StA_*$}
We will begin by describing some quotients
of the dual Steenrod algebra~$\StA_*$. For
any $n\geq1$, $(\zeta_1,\ldots,\zeta_n)\lhd\StA_*$
is a Hopf ideal so there is a quotient Hopf
algebra $\StA_*/(\zeta_1,\ldots,\zeta_n)$
together with the subHopf algebra
\[
\StP(n)_*
= \StA_*\square_{\StA_*/(\zeta_1,\ldots,\zeta_n)}\F_2
= \F_2[\zeta_1,\ldots,\zeta_n]\subseteq\StA_*
\]
and in fact
\[
\StA_*/\!/\StP(n)_* = \StA_*/(\zeta_1,\ldots,\zeta_n).
\]

Similarly, for any $s\geq0$, the ideal
$(\zeta_1^{2^s},\ldots,\zeta_n^{2^s})\lhd\StA_*$
is a Hopf ideal and there is a quotient
Hopf algebra
\[
\StA_*/\!/\StP(n)^{(s)}_*
= \StA_*/(\zeta_1^{2^s},\ldots,\zeta_n^{2^s})
\]
with associated subHopf algebra
\[
\StP(n)^{(s)}_*
= \StA_*\square_{\StA_*/\!/\StP(n)^{(s)}_*}\F_2
= \F_2[\zeta_1^{2^s},\ldots,\zeta_n^{2^s}]\subseteq\StA_*.
\]
For each $t\geq0$ there is a finite quotient
Hopf algebra
\[
\StP(n)^{(s)}_*/
(\zeta_1^{2^{s+t}},\zeta_2^{2^{s+t-1}},\ldots,
\zeta_t^{2^{s+1}},\zeta_{t+1}^{2^{s}},\ldots,\zeta_n^{2^s})
\]
and we have
\[
\StP(n)^{(s)}_* =
\lim_t \StP(n)^{(s)}_*/
(\zeta_1^{2^{s+t}},\zeta_2^{2^{s+t-1}},\ldots,
\zeta_t^{2^{s+1}},\zeta_{t+1}^{2^{s}},\ldots,\zeta_n^{2^s})
\]
where the limit is computed degree-wise. The
graded dual Hopf algebra
\[
\StP(n)_{(s)} =
(\StP(n)_*^{(s)})^*=\Hom(\StP(n)^{(s)}_*,\F_2)
\]
is the colimit of the finite dual Hopf algebras
\[
\Hom(\StP(n)^{(s)}_*/
(\zeta_1^{2^{s+t}},\zeta_2^{2^{s+t-1}},\ldots,
\zeta_t^{2^{s+1}},\zeta_{t+1}^{2^{s}},\ldots,\zeta_n^{2^s}),\F_2),
\]
i.e.,
\[
\StP(n)_{(s)} = \colim_t\Hom(\StP(n)^{(s)}_*/
(\zeta_1^{2^{s+t}},\zeta_2^{2^{s+t-1}},\ldots,
\zeta_t^{2^{s+1}},\zeta_{t+1}^{2^{s}},\ldots,\zeta_n^{2^s}),\F_2).
\]
Therefore $\StP(n)_{(s)}$ is a \Palgebra{} and
$\StP(n)^{(s)}_*$ is a $P_*$-algebra.


\section{Recollections on Ravenel's proof}
\label{sec:DCRproof}

We will assume the reader is familiar with
the strategy behind Ravenel's proof
of~\cite{DCR:Localn}*{theorem~3.9}. Since
we will be working in a very similar
situation with ring spectra which have
torsion free homology the main steps
will be applicable.

Following Ravenel~\cite{DCR:Localn}*{section~3},
we recall that there are compatible double
loop maps $\Omega\SU(2^s)\to B\U$ defined
using Bott's weak equivalence $\Omega\SU\to B\U$.
The corresponding virtual complex bundle on
$\Omega\SU(2^s)$ has Thom spectrum~$X_s$
which is an $\mathcal{E}_2$ ring spectrum.
Ravenel shows that
\[
\langle S^0\rangle=\langle X_1\rangle
> \langle X_2\rangle > \langle X_3\rangle
> \cdots > \langle X_s\rangle
> \langle X_{s+1}\rangle
> \cdots > \langle \MU \rangle.
\]
Of course, locally at the prime~$2$,
$\langle\MU\rangle=\langle BP\rangle$.

We are lead to consider the question:
How is $\langle\MSp\rangle$ related to
$\langle\MU\rangle$, and locally at~$2$,
to $\langle BP\rangle$? Since there are
maps of ring spectra $\MSp\to\MU\to BP$,
$\langle\MSp\rangle\geq\langle\MU\rangle=\langle BP\rangle$.
We will show that $\langle\MSp\rangle>\langle BP\rangle$
and also find an analogue of the sequence
of spectra~$X_s$ for~$\MSp$.

\begin{rem}\label{rem:RavenelProof}
Here are observations on the proof of theorem~3.1.

Rather than taking $M=M(4)$ to be the
mod~$4$ Moore spectrum we suggest using
the ring spectrum $M=F(M(2),M(2))$ (it is even
an $A_\infty$ ring spectrum in a canonical way).
Then $\langle M\rangle\geq\langle M(2)\rangle$
since the Moore spectrum $M(2)=S^0\cup_2e^1$
is an $M$-module spectrum, while
$\langle M(2)\rangle\geq\langle M\rangle$
since $M\sim M(2)\wedge DM(2)$. Therefore
$\langle M\rangle=\langle M(2)\rangle$.
Similar observations apply to the odd
prime case.

Also, for (a) and (b) the following observation
seems sufficient. Since
\begin{align*}
(X\wedge M)\wedge Z\sim * \nsim (Y\wedge M)\wedge Z
\quad&\Lra\quad
X\wedge(Z\wedge M)\sim * \nsim Y\wedge(M\wedge Z),  \\
\intertext{we have}
\langle Y\wedge M\rangle > \langle X\wedge M\rangle
\quad&\Lra\quad
\langle Y\rangle > \langle X\rangle.
\end{align*}
\end{rem}

\section{Some Thom spectra on loop spaces}
\label{sec:MSp-Y}

For background to this discussion see Mimura
\& Toda~\cite{MM&HT:TopLieGps}*{Chapter~IV,~\S6}.
The sequence of spaces making up the spectrum
of connected real $K$-theory include
\[
\underline{\kO}_4 \sim B\Sp,
\quad
\underline{\kO}_5 \sim \SU/\Sp,
\]
so there is a weak equivalence of infinite
loop spaces $\Omega\SU/\Sp\sim B\Sp$. At
the finite level we have compatible loop
maps
\[
\Omega\SU(2n)/\Sp(n)
\to \Omega\SU(2n+2)/\Sp(n+1)
\to \cdots \to B\Sp
\]
which define virtual symplectic bundles
on these spaces and we will denote the
Thom spectrum over $\Omega\SU(2^{s+1})/\Sp(2^s)$
by~$Y_s$. There is a map of $\mathcal{E}_1$
ring spectra $Y_s\to\MSp$ which induces
an injective ring homomorphism in
homology.

There are also compatible loop maps
$\Omega\SU(2n)\to\Omega\SU(2n)/\Sp(n)$,
and by pulling back from $\Omega\SU(2n)/\Sp(n)$
we obtain $\mathcal{E}_1$ ring spectra~$Y'_s$
over the $\Omega\SU(2^{s+1})$ together
with $\mathcal{E}_1$ maps $Y'_s\to Y_s$.

The mod~$2$ homology of $Y_s$ is given
by
\[
H_*(Y_s) = \F_2[y_1,y_2,\ldots,y_{2^s-1}],
\]
where $|y_k|=4k$. The induced homomorphism
$H_*(Y_s)\to H_*(\MSp)$ is actually an
isomorphism up to degree~$2^{s+2}-4$,
so we can choose the generators~$y_k$
so that they map to polynomial generators
of $H_*(\MSp)$. In particular, we assume
that each $y_{2^r-1}$ maps to the element
corresponding to the primitive
generator in $H_{2^{r+2}-4}(B\Sp_+)$
under the Thom isomorphism
$H_*(\MSp)\xrightarrow{\iso}H_*(B\Sp_+)$.
It is well known that the coaction
on these elements satisfies
\begin{equation}\label{eq:y(2^r-1)coaction}
\psi y_{2^r-1} =
\sum_{0\leq k\leq r}\zeta_k^4\otimes y_{2^{r-k}-1}^{2^{k}}.
\end{equation}

Consider the ideal
\[
J_s = (y_1,y_3,\ldots,y_{2^s-1}) \lhd H_*(Y_s).
\]
Then \eqref{eq:y(2^r-1)coaction} implies
that this is a subcomodule ideal, hence
the quotient ring $H_*(Y_s)/J_s$ is an
$\StA_*$-comodule algebra.

We will make use of the $\StA_*$-comodule
subalgebra
\[
\StP(s)^{(2)}_*=\F_2[\zeta_1^4,\zeta_2^4,\ldots,\zeta_s^4]\subseteq\StA_*
\]
and the quotient Hopf algebra $\StA_*/\!/\StP(s)^{(2)}_*$.

\begin{prop}\label{prop:ComoduleAlgebra}
The composition
\[
\xymatrix{
H_*(Y_s)\ar[r]^(.42)\psi\ar@/_15pt/[rr]^{\ph{F}}
& \StA_*\otimes H_*(Y_s)\ar[r]
& \StA_*\square_{\StA_*/\!/\StP(s)^{(2)}_*} H_*(Y_s)/J_s \\
}
\]
is an isomorphism of $\StA_*$-comodule
algebras, therefore
\[
H_*(Y_s) \iso
(\StA_*\square_{\StA_*/\!/\StP(s)^{(2)}_*}\F_2)
          \otimes H_*(Y_s)/J_s
\iso
\StP(s)^{(2)}_*\otimes H_*(Y_s)/J_s
\]
where $H_*(Y_s)/J_s$ has trivial coaction.
\end{prop}
In the limit this coincides with well-known
isomorphisms of $\StA_*$-comodule algebras
\begin{multline}\label{eq:H*MSp}
H_*(\MSp) \xrightarrow{\iso}
\StA_*\square_{\StA_*/\!/\StA_*^{(2)}} H_*(\MSp)/J
\\
\xrightarrow{\iso}
(\StA_*\square_{\StA_*/\!/\StA_*^{(2)}}\F_2)\otimes H_*(\MSp)/J
\xrightarrow{\iso}
\StA_*^{(2)} \otimes H_*(\MSp)/J
\end{multline}
where $J\lhd H_*(\MSp)$ is the comodule
ideal generated by the images of all
the~$y_{2^r-1}$.

Of course the quotient Hopf algebra
\[
\StA_*/\!/\StA_*^{(2)}
=
\F_2[\zeta_1,\zeta_2,\ldots]/(\zeta_1^4,\zeta_2^4,\ldots,\zeta_s^4)
\]
is also infinite dimensional and its dual
is the cyclic $\StA$-module $\StA_{(2)}=\StA/\!/\StE(2)$
where $\StE(2)\subseteq\StA$ is the subHopf
algebra generated by the elements~$\mathrm{P}^k_t$
($k\geq2$, $t\leq s$) in the notation
of~\cite{HRM:Book}*{section~XV.1}.

\section{Comparison of some Bousfield classes}
\label{sec:BousfieldClasses}

Now we can give our main topological results.
From now on we will work $2$-locally and omit
reference to this in notation, etc.

\begin{thm}\label{thm:Main}
The following inequalities for Bousfield
classes are satisfied:
\[
\langle S^0\rangle=\langle Y_1\rangle
> \langle Y_2\rangle
> \cdots > \langle Y_s\rangle
> \langle Y_{s+1}\rangle
> \cdots > \langle \MSp \rangle > \langle BP\rangle.
\]
\end{thm}
\begin{proof}
We will adopt the strategy of Ravenel's
proof for the analogous results on~$X_n$
and~$BP$. In particular we will make
repeated use of the Adams spectral
sequence. For $2$-local finite type
spectra~$X$ and~$Y$, by~\cite{JMB:CCSS}*{theorem~15.6}
there is a strongly convergent Adams
spectral sequence with
\begin{equation}\label{eq:ASS-cgce}
\mathrm{E}_2^{s,t}(X,Y) =
\Ext_{\StA}^{s,t}(H^*(Y),H^*(X))
\Lra [X,Y\sphat_2\;]^{s-t}.
\end{equation}
We can dualise to homology and use comodules
to obtain another interpretation of the
$\Ext$ group:
\begin{equation}\label{eq:ASS-cgce-homology}
\mathrm{E}_2^{s,t}(X,Y) =
\Coext_{\StA_*}^{s,t}(H_*(X),H_*(Y))
\Lra [X,Y\sphat_2\;]^{s-t}.
\end{equation}

Of course we know that $BP$ is an $\MSp$-module
spectrum so
$\langle \MSp \rangle \geq \langle BP\rangle$.
To prove the strict inequality we follow Ravenel
and reduce to showing the triviality of the
$\mathrm{E}_2$-term of the Adams spectral
sequence
\[
\mathrm{E}_2^{s,t}(BP,\MSp) =
\Coext_{\StA_*}^{s,t}(H_*(BP),H_*(\MSp))
\Lra [BP,\MSp\sphat_2\;]^{s-t}.
\]
By~\eqref{eq:H*MSp},
\[
H_*(\MSp) \iso
(\StA_*\square_{\StA_*/\!/\StA_*^{(2)}}\F_2)\otimes H_*(\MSp)/J
\]
where $H_*(\MSp)/J$ has trivial
$\StA_*/\!/\StA_*^{(2)}$-coaction. Combining
this with~\eqref{eq:H*BP} we have
\begin{align*}
\mathrm{E}_2^{s,*}(BP,\MSp) &\iso
\Coext_{\StA_*}^{s,*}(H_*(BP),\StA_*\square_{\StA_*/\!/\StA_*^{(2)}}\F_2)
\otimes H_*(\MSp)/J   \\
&\iso \Coext_{\StA_*/\!/\StA_*^{(2)}}^{s,*}(\StA_*^{(1)},\F_2)
\otimes H_*(\MSp)/J.
\end{align*}

To compute the $\Coext$ term here we will
dualise and use a Cartan-Eilenberg spectral
sequence of the form~\eqref{eq:CESS-algebras}
for
\[
\Ext_{(\StA_*/\!/\StA_*^{(2)})^*}(\F_2,\StA_{(1)})
\iso
\Coext_{\StA_*/\!/\StA_*^{(2)}}^{s,*}(\StA_*^{(1)},\F_2).
\]
We will base this on the sequence of algebras
\[
\xymatrix{
(\StA_*/\!/\StA_*^{(1)})^*\ar[r] &
(\StA_*/\!/\StA_*^{(2)})^*\ar[r] &
(\StA_*^{(1)}/\!/\StA_*^{(2)})^*
}
\]
which is dual to
\[
\StA_*^{(1)}/\!/\StA_*^{(2)}
\to
\StA_*/\!/\StA_*^{(2)}
\to
\StA_*/\!/\StA_*^{(1)}.
\]
This spectral sequence is tri-graded with
$\mathrm{E}_2$-term
\[
\leftidx{^{\text{\tiny CE}}}{\mathrm{E}}{_2^{s,t}}
=
\Ext^s_{(\StA_*^{(1)}/\!/\StA_*^{(2)})^*}(\F_2,
\Ext_{(\StA_*/\!/\StA_*^{(1)})^*}^t(\F_2,\StA_{(1)}))
\]
where we have suppressed mention of the
internal grading. As $(\StA_*/\!/\StA_*^{(1)})^*$
acts trivially on $\StA_{(1)}$,
\[
\Ext_{(\StA_*/\!/\StA_*^{(1)})^*}^t(\F_2,\StA_{(1)})
\iso
\Ext_{(\StA_*/\!/\StA_*^{(1)})^*}^t(\F_2,\F_2)
\overset{(\StA_*^{(1)}/\!/\StA_*^{(2)})^*}\boxtimes
\StA_{(1)},
\]
where $\overset{(\StA_*^{(1)}/\!/\StA_*^{(2)})^*}\boxtimes$
indicates the $\k$-tensor product with the diagonal
action of the cocommutative Hopf algebra
$(\StA_*^{(1)}/\!/\StA_*^{(2)})^*$. Since there
is a surjection of Hopf algebras
\[
\StA_*^{(1)}\to \StA_*^{(1)}/\!/\StA_*^{(2)},
\]
dually there is an injection
\[
(\StA_*^{(1)}/\!/\StA_*^{(2)})^*\to
(\StA_*^{(1)})^* = \StA_{(1)}
\]
so the Milnor-Moore theorem implies that
$\StA_{(1)}$ is a free
$(\StA_*^{(1)}/\!/\StA_*^{(2)})^*$-module.
Hence as $(\StA_*^{(1)}/\!/\StA_*^{(2)})^*$-modules,
\[
\Ext_{(\StA_*/\!/\StA_*^{(1)})^*}^t(\F_2,\F_2)
\overset{(\StA_*^{(1)}/\!/\StA_*^{(2)})^*}\boxtimes
\StA_{(1)}
\iso
\StA_{(1)}
\otimes
\Ext_{(\StA_*/\!/\StA_*^{(1)})^*}^t(\F_2,\F_2).
\]
Feeding this into our $\mathrm{E}_2$-term we obtain
\begin{align*}
\leftidx{^{\text{\tiny CE}}}{\mathrm{E}}{_2^{s,t}}
&\iso
\Ext^s_{(\StA_*^{(1)}/\!/\StA_*^{(2)})^*}(\F_2,\StA_{(1)})
\otimes
\Ext_{(\StA_*/\!/\StA_*^{(1)})^*}^t(\F_2,\F_2) \\
&\iso
\Coext^s_{\StA_*^{(1)}/\!/\StA_*^{(2)}}(\StA_*^{(1)},\F_2)
\otimes
\Ext_{(\StA_*/\!/\StA_*^{(1)})^*}^t(\F_2,\F_2) \\
&\iso
\Coext^s_{\StA_*/\!/\StA_*^{(1)}}(\StA_*,\F_2)
\otimes
\Ext_{(\StA_*/\!/\StA_*^{(1)})^*}^t(\F_2,\F_2) \\
&= 0
\end{align*}
since $\StA_*$ is a cofree comodule over the
$P_*$-algebra $\StA_*/\!/\StA_*^{(1)}$.

Now we can conclude that $\mathrm{E}_2^{s,*}(BP,\MSp) = 0$
and so $(\MSp\sphat_2)^*(BP)=0$.

%
%

The rest of the proof is similar and uses
the spectral sequence~\eqref{eq:CESS-algebras}
to reduce to calculate some~$\Coext$ groups
using the dual algebras. We will describe
the main details required for these.

To verify that $\langle Y_n\rangle>\langle\MSp\rangle$
the key step involves showing that
\[
\Coext_{\StA_*}^{*,*}(\StA^{(2)}_*,\StP(n)^{(2)}_*)
=
\Coext_{\StA_*}^{*,*}(\StA^{(2)}_*,\StA_*\square_{\StA_*/\!/\StP(n)^{(2)}_*}\F_2)
= 0,
\]
where
\[
\Coext_{\StA_*}^{*,*}(\StA^{(2)}_*,\StA_*\square_{\StA_*/\!/\StP(n)^{(2)}_*}\F_2)
\iso
\Coext_{\StA_*/\!/\StP(n)^{(2)}_*}^{*,*}(\StA^{(2)}_*,\F_2).
\]
By \cite{JPM&KP:MoreConcise}*{corollary 21.2.5}, there is
a sequence of Hopf algebras
\[
\StA^{(2)}_*/\!/\StP(n)^{(2)}_*
=(\StA_*/\!/\StA^{(2)}_*) \backslash\!\backslash (\StA_*/\!/\StP(n)^{(2)}_*)
\to \StA_*/\!/\StP(n)^{(2)}_* \to \StA_*/\!/\StA^{(2)}_*
\]
whose the dual sequence of algebras gives
rise to a Cartan-Eilenberg spectral
sequence~\eqref{eq:CESS-algebras} with
\begin{align*}
\leftidx{^{\text{\tiny CE}}}{\mathrm{E}}{_2^{s,t}}
&\iso
\Ext^s_{\StA_{(2)}}(\F_2,
\Ext_{(\StA_*/\!/\StA_*^{(2)})^*}^t(\F_2,\StA_{(2)}))  \\
&\iso
\Ext^s_{\StA_{(2)}}(\F_2,
\Ext_{(\StA_*/\!/\StA_*^{(2)})^*}^t(\F_2,\F_2)
\overset{\StA_{(2)}}\boxtimes\StA_{(2)})  \\
&\iso
\Ext^s_{\StA_{(2)}}(\F_2,\StA_{(2)})
\otimes\Ext_{(\StA_*/\!/\StA_*^{(2)})^*}^t(\F_2,\F_2))  \\
&= 0
\end{align*}
since $\StA_{(2)}$ is a \Palgebra.


Finally, the verification that
$\langle Y_n\rangle>\langle Y_{n+1}\rangle$ proceeds
in this fashion with $\StP(n+1)^{(2)}_*$ in place
of $\StA_*^{(2)}$ and involves the vanishing of
\[
\Coext_{\StP(n+1)^{(2)}_*/\!/\StP(n)^{(2)}_*}^{s,4*}(\StP(n+1)^{(2)}_*,\F_2)
\iso
\Coext_{\StP(n+1)_*/\!/\StP(n)_*}^{s,*}(\StP(n+1)_*,\F_2)
\]
which follows from the facts that $\StP(n+1)_*$
is cofree over $\StP(n+1)_*/\!/\StP(n)_*$ by
the Milnor-Moore theorem and the latter is
a $P_*$-algebra.
\end{proof}

As an exercise, the reader may like to rederive
Ravenel's results for $BP$ and the $X_n$'s using
our approach.

There are some other interesting spectra that
intermediate between $\MSp$ and $BP$. These
include Pengelley's $BoP$~\cite{DJP:BoP},
$\MSU$ and the sequence of $\MSp$-module
spectra obtained by killing the Ray elements
$\phi_{(s)}=\phi_{s}\in\pi_{2^{s+2}-3}(\MSp)$
as described by Botvinnik~\cite{BIB:MfdsSingANSS}
(both including and excluding the generator
of $\pi_1(\MSp)$). In fact if we kill finitely
many of the Ray elements this does not change
the Bousfield class since they are all nilpotent,
whereas if we kill all of them we get a spectrum
whose Bousfield class is the same as~$BP$;
similarly, $\langle\MSU\rangle=\langle BP\rangle$.
So there are no novel Bousfield classes between
$\langle\MSp\rangle$ and $\langle BP\rangle$
stemming from these.

\end{document}